\documentclass[hidelinks,11pt]{amsart}
\renewcommand*\contentsname{Table of Content}

\usepackage{amsfonts,amsmath,amssymb,amsthm}
\usepackage[T1]{fontenc}
\usepackage{lmodern}
\usepackage{textcomp}
\usepackage{anyfontsize}
\usepackage{mathtools}
\usepackage{cancel}
\usepackage{graphicx} 
\usepackage{tikz-cd}
\usepackage{quiver}
\usepackage{spectralsequences}
\usepackage{tabularx}
\usepackage{soul}
\usepackage{subfig}
\usepackage[all]{xy}
\usepackage[colorlinks=true]{hyperref}
\hypersetup{
    colorlinks,
    linkcolor={SecBlue},
    citecolor={PeoBlue},
    urlcolor={blue!80!black}
}

\definecolor{DarkBlue}{rgb}{.1, 0.35, 0.6} 
\definecolor{PeoBlue}{rgb}{.15, 0.8, 0.6} 
\definecolor{SecBlue}{rgb}{.15, 0.4, 0.4} 
\definecolor{DarkBrown}{rgb}{.5, 0.2, 0.2} 

\usepackage{cleveref}
\usepackage{spectralsequences}
\newtheorem{theorem}{Theorem}
\numberwithin{theorem}{section} 

\newtheorem{lem}[theorem]{Lemma}

\newtheorem{remark}[theorem]{Remark}
\newtheorem{theo}[theorem]{Theorem}

\theoremstyle{definition}
\newtheorem{prop}[theorem]{Proposition}
\usepackage{placeins}

\usepackage{tikz-cd}
\newsavebox{\pullback}
\sbox\pullback{%
\begin{tikzpicture}%
\draw (0,0) -- (1ex,0ex);%
\draw (1ex,0ex) -- (1ex,1ex);%
\end{tikzpicture}}

\usepackage{setspace}
\usepackage{newtxtext, newtxmath}
\usepackage{enumitem}
\usepackage{pgfplots}
\usepackage{esint}
\graphicspath{ {./images/} }
\usepackage{changepage}   
\usepackage{caption}
\usepackage{float}
\usepackage{circledsteps}

\usepackage{multicol}

\newcommand{\Fbb}{\mathbb{F}}
\newcommand{\Sbb}{\mathbb{S}}
\newcommand{\Gbb}{\mathbb{G}}
\newcommand{\ctwo}{\operatorname{E}^{hC_2}}
\newcommand{\csix}{{\operatorname{E}^{hC_6}}}
\newcommand{\csixv}{\operatorname{E}^{hC_6} \wedge V(0)}
\newcommand{\csixy}{\operatorname{E}^{hC_6} \wedge Y}
\newcolumntype{L}{>{$}l<{$}}

\makeatletter
\newcommand{\obullet}{\mathbin{\mathpalette\make@circled\bullet}}
\newcommand{\make@circled}[2]{%
  \ooalign{$\m@th#1\smallbigcirc{#1}$\cr\hidewidth$\m@th#1#2$\hidewidth\cr}%
}
\newcommand{\smallbigcirc}[1]{%
  \vcenter{\hbox{\scalebox{0.77778}{$\m@th#1\bigcirc$}}}%
}

\newcommand {\R}{\mathbb{R}}
\newcommand {\Z}{\mathbb{Z}}
\newcommand {\C}{\mathbb{C}}
\newcommand {\F}{\mathbb{F}}

\newcommand {\W}{\mathbb{W}}

\newcommand {\E}{\operatorname{E}}

\pgfplotsset{compat=1.18}

\title{The $\operatorname{E}^{hC_6}_2$-Homology of $\mathbb{R}P^2$ and $\mathbb{R}P^2 \wedge \mathbb{C}P^2$}

\author[Bobkova]{Irina Bobkova}
\author[Carlisle]{Jack Carlisle}
\author[Fitz]{Emmett Fitz}
\author[Ji]{Mattie Ji}
\author[Kilway]{Peter Kilway}
\author[Kim]{Hillary Kim}
\author[O'Neal]{Kolton O'Neal}
\author[Schuckman]{Jacob Schuckman}
\author[Tilton]{Scotty Tilton}

\begin{document}
\begin{abstract}
Let $\E_2$ be the Morava E-theory of height $2$ at the prime $2$. In this paper, we compute the homotopy groups of ${\E_2^{hC_6}}\wedge \R P^2$ and ${\E_2^{hC_6}} \wedge\R P^2 \wedge \C P^2$ using the homotopy fixed point spectral sequences.
\end{abstract} 

\maketitle

\setcounter{tocdepth}{1}
\renewcommand\contentsname{Table of Contents}
\tableofcontents

\section{Introduction}
One of the fundamental questions of algebraic topology is the computation of stable homotopy groups $\pi_k(\Sbb)$ of the sphere spectrum $\Sbb$ for $k > 0$. A classical theorem of Serre asserts that $\pi_k(\Sbb)$ is a finite abelian group for all $k > 0$, and thus the stable homotopy groups of spheres can be studied prime by prime. By the chromatic convergence theorem of Hopkins and Ravenel (\cite{chromatic_convergence}, Theorem 7.5.7) any $p$-local finite spectrum $X$ is the homotopy limit of the \emph{chromatic tower}:
\[\begin{tikzcd}
	X & {...} & {L_{n} X}& {...} & {L_{1} X} & {L_{0} X}
	\arrow[from=1-1, to=1-2]
	\arrow[from=1-2, to=1-3]
	\arrow[from=1-3, to=1-4]
	\arrow[from=1-4, to=1-5]
	\arrow[from=1-5, to=1-6]
\end{tikzcd},\]
where $L_n$ denotes the localization with respect to a wedge of Morava $K$-theories $\bigvee\limits_{i=0}^n K(i)$ at the prime $p$. 

Here, the connecting morphism from $L_{n} X$ to $L_{n-1} X$ is given by the \emph{chromatic fracture square}:
\[\begin{tikzcd}
L_{n} X \arrow[r] \arrow[d] 
\arrow[dr, phantom, "\usebox\pullback" , very near start, color=black]
& L_{K(n)} X \arrow[d] \\
L_{n-1} X \arrow[r] & L_{n-1} L_{K(n)} X \\
\end{tikzcd}.\]
Thus, to study the stable homotopy groups of spheres, it is important for us to understand the intermediate terms $L_{K(n)} \Sbb_{(p)}$ for the $p$-local sphere spectrum $\Sbb_{(p)}$ at all primes $p$. 

Let $\Sbb_n$ be the $n$-th Morava stabilizer group at $p$, $\Gbb_n$ be the $n$-th extended Morava stabilizer group at $p$, and $\E_n$ be the Morava E-theory of height $n$ at a prime $p$. A celebrated result of Devinatz and Hopkins \cite{DEVINATZ20041} showed that $L_{K(n)} \Sbb_{(p)} \simeq \E_n^{h\mathbb{G}_n}$ and there is a homotopy fixed point spectral sequence (HFPSS) with signature:
\[E_2^{*,*}: H_c^*(\Gbb_n, \pi_*(\E_n)) \Longrightarrow \pi_{*} (\E_n^{h\mathbb{G}_n}) \cong \pi_{*}(L_{K(n)} \mathbb{S}_{(p)}).\]
Furthermore, for any closed subgroup $G \subseteq \Gbb_n$, this spectral sequence descends down to a HFPSS:
\begin{equation}\label{eq:closed-subgroups}
E_2^{*,*}: H^*_c(G, \pi_*(\E_n)) \Longrightarrow \pi_{*}(\E_n^{hG}).
\end{equation}

This result  may be extended  from $\mathbb{S}$ to any finite complex. Let $X$ be a finite complex and $G$ be a closed subgroup of $\Gbb_n$. Then there exists a HFPSS with signature:
    \begin{equation}\label{eq:2}
    E^{*,*}_2 = H^*_c(G, (\E_n)_* X) \Longrightarrow \pi_{*}(\E_n^{hG} \wedge X).      
    \end{equation}

The HFPSS \eqref{eq:closed-subgroups} is a multiplicative spectral sequence, and there is a natural map of spectral sequences induced by the unit map $\E_n^{hG} \to \E_n^{hG} \wedge X$. This gives the HFPSS \eqref{eq:2} a module structure over \eqref{eq:closed-subgroups}. In general, this gives a multiplicative Leibnitz rule on the HFPSS for $\E_n^{hG}$ and a ``module'' Leibnitz rule on the HFPSS for $\E_n^{hG} \wedge X$.

Understanding $\E_n^{hG}$ for finite subgroups $G$ has been crucial in demystifying the structure of $\E_n^{h\mathbb{G}_n}$.
For example, at $n=2$ there exist resolutions that provide a decomposition of $\E_2^{h\mathbb{G}_2}$ in terms of $E_2^{hG}$ for various finite subgroups $G$ \cite{ghmr, BehrensQ, henn_res, henn_centr, Bobkova_2018}.

This provides a motivation for studying the spectra that appear in these resolutions. The focus of this paper is on the spectrum $\E_2^{hC_6}$ at the prime $2$. Our goal is to compute the $\E_2^{hC_6}$-homology of $\mathbb{R} P^2$ and $\mathbb{R} P^2\wedge \mathbb{C} P^2$, by determining the differentials and extensions in their respective homotopy fixed point spectral sequences.

 Let $V(0)$ be the cofiber of multiplication by 2 on $\Sbb$, and let $Y$ be the smash product of $V(0)$ with $C_{\eta}$, the cofiber of the stable Hopf map $\eta$. Then we have that 
\[ V(0) \simeq \Sigma^{-1}\Sigma^{\infty}\R P^2 \quad\text{and}\quad Y \simeq \Sigma^{-3}\Sigma^{\infty}(\R P^2 \wedge \C P^2).\]
Therefore, computing the $\E_2^{hC_6}$-homology of $\R P^2$ and $\R P^2 \wedge \C P^2$ is equivalent to computing the $\E_2^{hC_6}$-homology of $V(0)$ and $Y$. The goal of this paper is to completely compute the HFPSS for $\E_2^{hC_6} \wedge V(0)$ and $\E_2^{hC_6} \wedge Y$ in Sections \ref{sec:c6v0} and \ref{sec:c6y}. Along the way, as a preliminary step, we will also review the HFPSS for $\E_2^{hC_2}$, $\E_2^{hC_2} \wedge V(0)$ and $\E_2^{hC_6}$ in Sections \ref{sec:c2}, \ref{sec:c2v0} and \ref{sec:c6}.

\subsection*{Acknowledgements} This work arose from the 2024 electronic Computational Homotopy Theory (eCHT) Research Experience for Undergraduates, supported by NSF grant DMS-2135884. The authors thank Dan Isaksen for his support and leadership within the eCHT.  The first named author thanks Agnes Beaudry for inspiring this project and for the many related discussions over the years.

\section{The homotopy fixed point spectral sequence for \texorpdfstring{$\operatorname{E}_2^{hC_2}$}{E2hC2}}\label{sec:c2}

In the rest of the paper, we will write $\E= \E_2$  to denote the Morava E-theory of height $2$ at the prime $2$. The homotopy fixed point spectral sequence
\begin{equation}\label{eq:C2-SS}
E^{s,t}_r(\E^{hC_2}): H^s(C_2, \E_t) \Longrightarrow \pi_{t-s} \E^{hC_2} 
\end{equation} 
has been completely computed and studied for decades, see, for example, \cite{HeardStojanoska} and \cite{Hahn_Shi_2020} for more recent accounts. In this section, we will summarize the key results of this computation without proof as we establish notation that we will use throughout the paper. Note that we use the Adams notation for our spectral sequences: when we refer to elements in $E_r^{s, t}$, we mean elements in stem $t-s$ and filtration $s$ in the homotopy fixed point spectral sequence.

\begin{figure}
\centering
\includegraphics[width=0.65\textwidth]{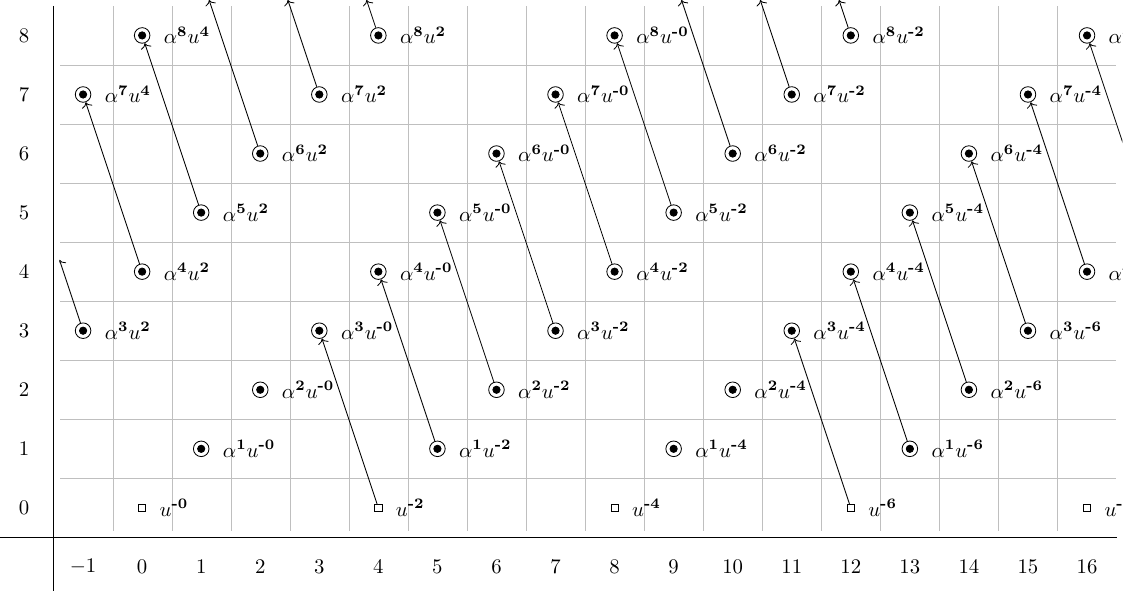}
\includegraphics[width=0.65\textwidth]{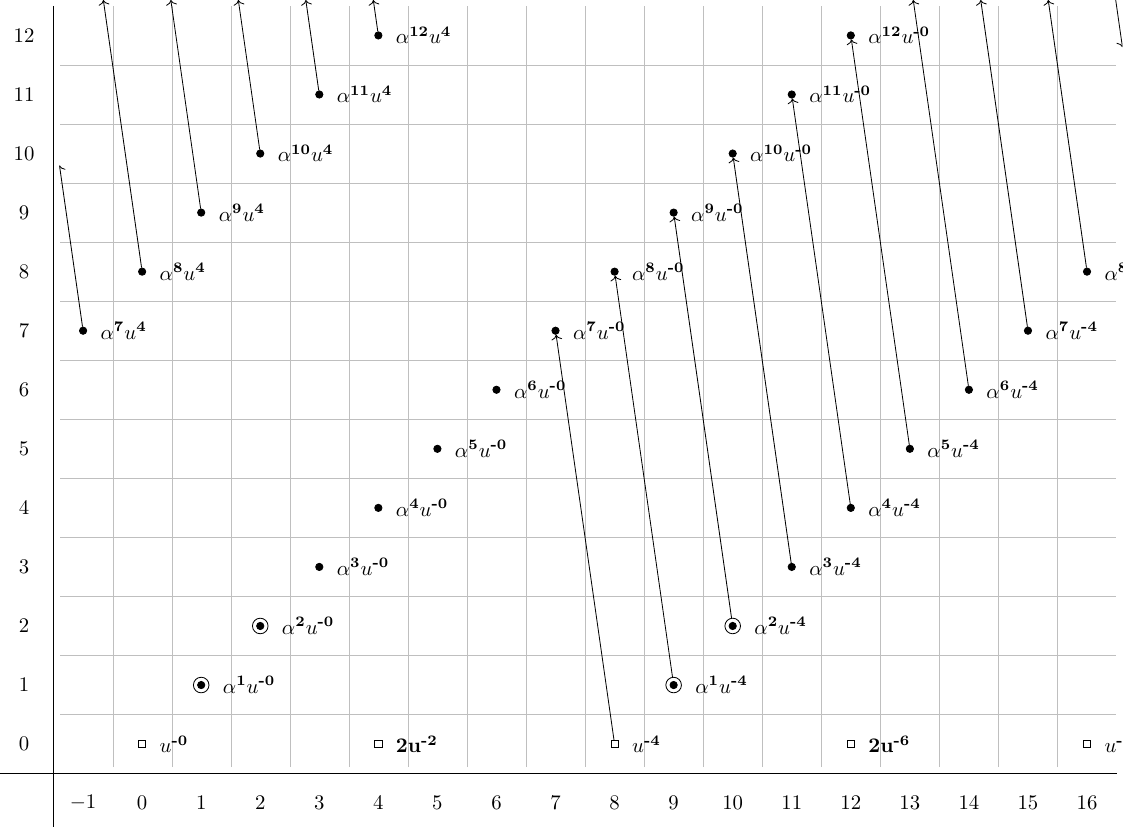}
\caption{The $E_3$ (top) and $E_7$ (bottom) page of the HFPSS for $\ctwo$. The notation is as follows: $\protect\obullet=\mathbb{F}_4\llbracket u_1\rrbracket$, $\bullet=\Fbb_4$, and 
   $\Box=\mathbb{W}\llbracket u_1 \rrbracket$.}
   \label{fig:c2_e3_e7_page}
\end{figure}

Let $\W = W(\Fbb_4)$ be the ring of Witt vectors for $\Fbb_4$. Recall that 
\[\E_* =  \W \llbracket  u_1\rrbracket  [u^{\pm 1}], \quad |u_1| = 0 \text{ and } |u|=-2\]
and $\mathbb{G}_2 \cong (W(\mathbb{F}_4)\langle S\rangle /Sa^{\sigma}=aS, S^2=2)^{\times}$, where $\sigma$ denotes the lift of the Frobenius morphism. 
The central subgroup $C_2 = \{\pm 1\} \subset \mathbb{G}_2$ acts trivially on $E_0 = \W\llbracket  u_1\rrbracket  $ and by multiplication
by $-1$ on $u$. With this action,
\[
E_2^{*,*} = H^* (C_2 , \E_* ) = \W\llbracket  u_1\rrbracket  [[u^2]^{\pm 1} , \alpha]/(2\alpha),\]
where $\alpha \in H^1(C_2 , \pi_2(\E))$  is the image of the generator of $H^1(C_2 , \Z[sgn])$ under the map
which sends the generator of the sign representation $\Z[sgn]$ to $u^{-1}$.

\begin{lem}\label{lem::c2_d3}
The $d_3$ differentials in \eqref{eq:C2-SS} are generated by
\[d_3(u^{-2}) = \alpha^3 u_1 \]  and 
linearity with respect to $\alpha$, $u_1$ and $u^{\pm 4}$.
\end{lem}

\begin{lem}\label{lem::c2_d7}
The $d_7$ differentials in \eqref{eq:C2-SS}  are generated by
\[d_7(u^{-4}) = \alpha^7 \] and
linearity with respect to $\alpha$ and $u^{\pm 8}$. 
\end{lem}
We have $E^{*,*}_8(\E^{hC_2}) = E^{*,*}_{\infty}(\E^{hC_2})$, displayed in \Cref{fig:c2_einf_page}.

\begin{figure}
\includegraphics[width=0.8\linewidth]{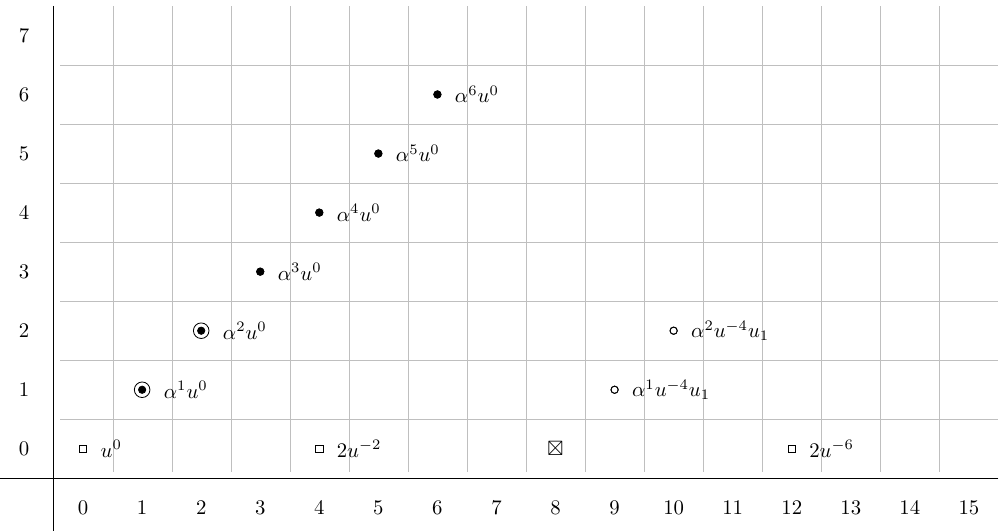}
    \caption{The $E_8 = E_{\infty}$ page of the HFPSS for $\ctwo$. 
     The notation is as follows: $\protect\obullet=\mathbb{F}_4\llbracket u_1\rrbracket$, $\bullet=\Fbb_4$, $\circ=u_1 \mathbb{F}_4\llbracket u_1 \rrbracket$, 
   $\Box=\mathbb{W}\llbracket u_1 \rrbracket$, and the $\boxtimes$ at $(8,0)$ denotes $2u^{-4}W\oplus u^{-4}u_1W\llbracket u_1\rrbracket$. The homotopy groups are 16-periodic, with periodicity generator $u^{-8}$.}
    \label{fig:c2_einf_page}
\end{figure}

\section{The homotopy fixed point spectral sequence for \texorpdfstring{$\E_2^{hC_2} \wedge V(0)$}{E2hC2 sm V(0)}}\label{sec:c2v0}
In this section, we will compute the spectral sequence
\begin{equation}\label{eq:C2-V0-SS}
E^{s,t}_r(\E^{hC_2} \wedge V(0)): H^s(C_2, \pi_t(E\wedge V(0))) \Longrightarrow \pi_{t-s} \E^{hC_2} \wedge V(0) 
\end{equation} 
Our starting point is the  fiber sequence of spectra
\begin{equation}\label{eq:fiber-seq-V(0)}
   \E \xrightarrow{2} \E\xrightarrow{i} \E\wedge V(0) \xrightarrow{p} \Sigma \E.
\end{equation}
The map $i$ is the inclusion of the bottom cell of $V(0)$, and the map $p$ is the projection to the top cell. 

The fiber sequence \eqref{eq:fiber-seq-V(0)} induces a short exact sequence of homotopy groups
\begin{equation}\label{eq:ses-1}
 0 \rightarrow \pi_t \E\xrightarrow{2}\pi_t \E\rightarrow\pi_t(\E\wedge V(0)) \cong \pi_t E /2 \rightarrow 0
\end{equation}
for any $t$ (if $t$ is odd, every term in this sequence is zero). Note also that we have \[\pi_*(\E\wedge V(0)) = \E_*/2 \cong \F_4\llbracket u_1 \rrbracket[u^{\pm 1}].\] 
 
 The short exact sequence \eqref{eq:ses-1} induces a long exact sequence in group cohomology 
 \begin{equation}\label{eq:1}
\cdots\to H^s(C_2, \pi_t \E)\xrightarrow{2} H^s(C_2, \pi_t \E)\xrightarrow{i} H^s(C_2,\pi_t \E/2)\xrightarrow{p} H^{s+1}(C_2, \pi_t \E) \to \cdots 
 \end{equation}

\begin{lem}[Section 2.2 of \cite{1a7d99963b1048d8ad18fe824397caab}]\label{lem:E2C2V0}
The $E_2$-page of the HFPSS for $\ctwo \wedge V(0)$ is
\[ E_2^{*,*}(\E^{hC_2}\wedge V(0))=H^* (C_2 , \E_*/2) 
= \F_4 \llbracket u_1 \rrbracket[u^{\pm 1} ][\alpha].\]
 Note that $H^* (C_2 , E_* /2)$ is a module over $H^* (C_2 , E_*)$
and
$\eta = \alpha u_1$.
\end{lem}
\begin{remark}
Let $v_1 \in \pi_2(V(0))$ be an element which maps to $\eta \in \pi_1(\Sbb)$ under the map $p$ and consider the fate of $v_1$ in the commutative diagram, where the vertical maps are unit maps of the ring spectrum $\ctwo$:
$$
\xymatrix{
\ldots \ar[r] &\pi_2(\Sbb) \ar[d] \ar[r]^i &\pi_2(V(0)) \ar[d] \ar[r]^{p} &\pi_1(\Sbb) \ar[d] \ar[r]^{2} &\ldots\\
\ldots \ar[r] &\pi_2(\E^{hC_2}) \ar[r]^-{i} &\pi_2(\ctwo \wedge V(0)) \ar[r]^-{p} &\pi_1(\ctwo) \ar[r]^{2} &\ldots
}$$
Since $\eta \in \pi_1(\ctwo)$ is an element of order 2 and is detected by $u_1\alpha = u_1 u^{-1}h$, we have that $v_1\in \pi_2(\ctwo \wedge V(0))$ is detected by $u_1 u^{-1}$. 

\end{remark}

\begin{figure}[h]
    \centering
    \includegraphics[width=0.7\linewidth]{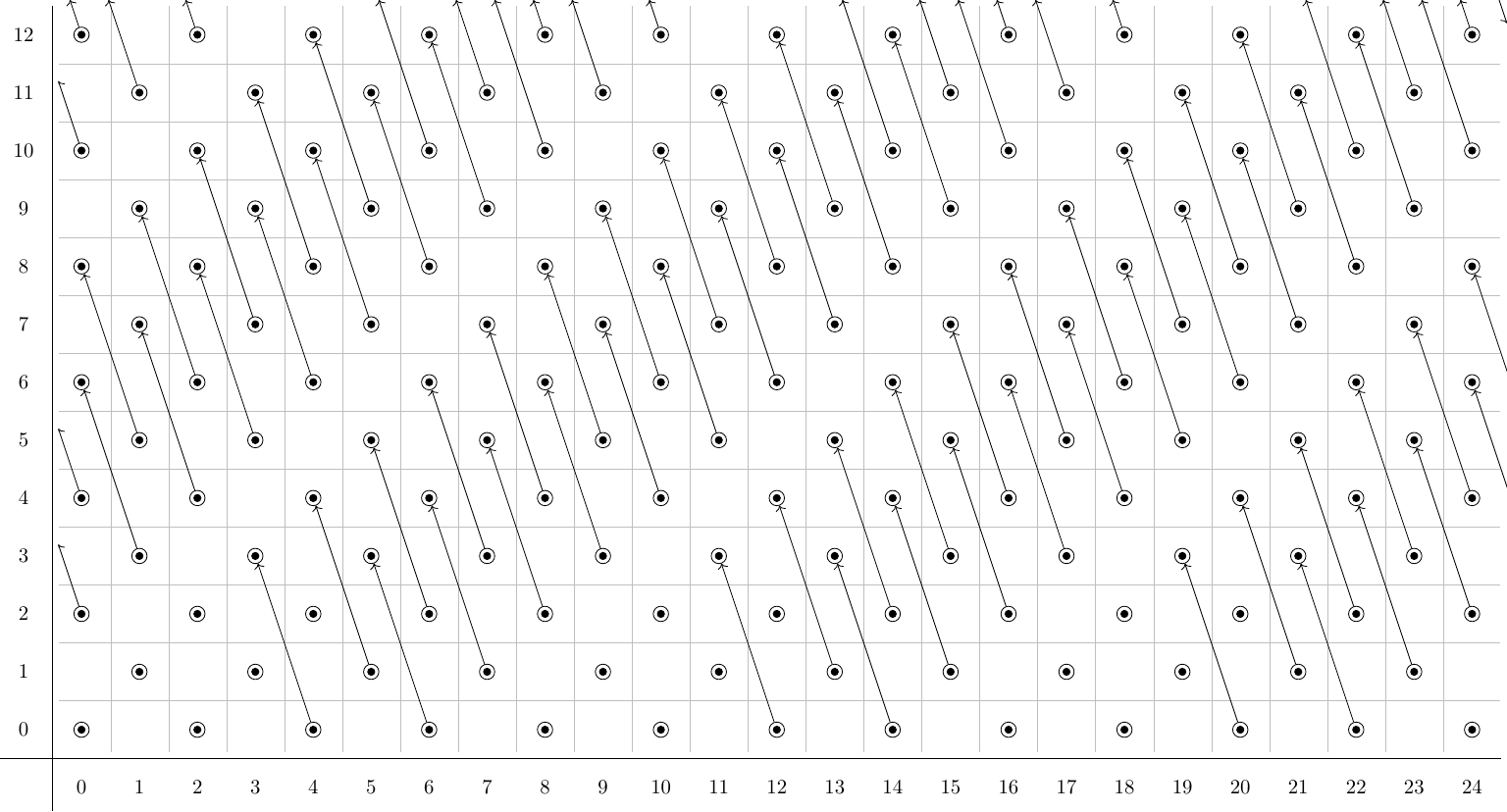}
    \caption{The $E_3$ page of the HFPSS for $\ctwo \wedge V(0)$. The symbol $\protect\obullet$ represents $\mathbb{F}_4\llbracket u_1\rrbracket$.}
    \label{fig:E3_page}
\end{figure}

\subsection{The \texorpdfstring{$d_3$}{d3}-differentials}

\begin{lem}\label{lem:d3c2v0}
The $d_3$ differentials in the homotopy fixed point spectral sequence \eqref{eq:C2-V0-SS} for $\ctwo \wedge V(0)$ are generated by
\begin{align*}
    d_3(u^{-2})&=\alpha^3 u_1,\\
    d_3(u^{-3})&=\alpha^3 u^{-1} u_1,
\end{align*}
and linearity with respect to 
$\alpha$, $u_1$ and $u^{\pm 4}$. 
\end{lem}

\begin{proof}
The first differential and linearity with respect to $u_1$, $\alpha$ and $u^{\pm 4}$ follow from  \Cref{lem::c2_d3} and the fact that this spectral sequence is a module over $E_r(\E^{hC_2})$.
Next we note that $d_3(v_1^3)=v_1\eta^3$ in the Adams--Novikov spectral sequence for $\pi_*V(0)$ (see, for example, \cite[Thm. 5.13(a)]{ravnovice}). Then we have
\[
d_3(v_1^3)=d_3(u^{-3}u_1^3)=u_1^3 d_3(u^{-3})=v_1 \eta^3= u_1 u^{-1} \alpha^3 u_1^3,
\]
implying $d_3(u^{-3})=\alpha^3u^{-1} u_1$. 
\end{proof}

All $d_3$ differentials are injective, so the sources of these differentials vanish on the $E_4$-page. The differentials are not surjective, and their cokernels are copies of $\mathbb{F}_4$ generated by powers of $\alpha$. This can be seen in \Cref{fig:E7_page}.

\begin{figure}[h]
\centering
\includegraphics[width=0.8\linewidth]{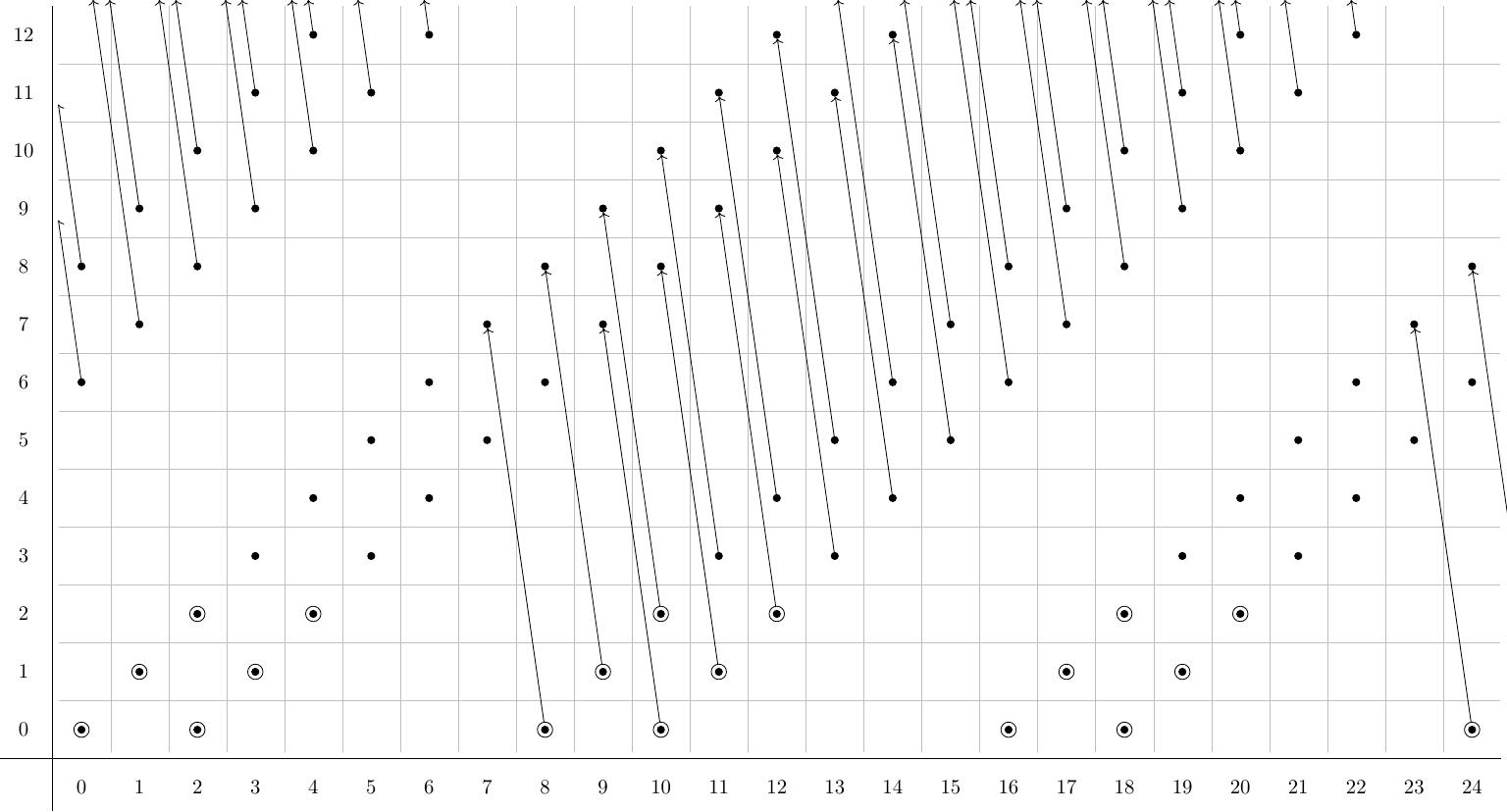}
\caption{The $E_7$ page of the HFPSS for $\ctwo \wedge V(0)$. The symbol $\bullet$ represents $\Fbb_4$, and $\protect\obullet$ represents $\mathbb{F}_4\llbracket u_1\rrbracket$.}
\label{fig:E7_page}
\end{figure}

\subsection{The \texorpdfstring{$E_\infty$}{E infinity} page}

Due to sparseness, there are no possible $d_{r}$ differentials for even $r$. There are possible $d_5$ differentials, but they are trivial.
\begin{lem}\label{lem::c2_d5}
There are no nontrivial $d_5$ differentials in the spectral sequence \eqref{eq:C2-V0-SS}. 
\end{lem}

\begin{proof}
This follows from the module structure of $E_r(\E^{hC_2} \wedge V(0))$ over $E_r(\E^{hC_2})$ and the fact that there are no $d_5$ differentials in  $E_r(\E^{hC_2})$.
\end{proof}

Before we proceed with the next differentials, we would like to state the following useful technical lemma, which is a special case of the Geometric Boundary Theorem \cite[Thm. 2.3.4]{ravgreen}.

\begin{lem}\label{GBT}
There are maps 
$
\delta_r: E_r^{s,t}(\ctwo \wedge V(0)) \to E_r^{s+1, t} (\ctwo)
$
such that 
\[\delta_2: E_2^{s,t}(\ctwo \wedge V(0)) \to E_2^{s+1, t} (\ctwo)\]
is the connecting homomorphism arising from \eqref{eq:fiber-seq-V(0)}. For all $r$, 
\[
\delta_r d_r = d_r\delta_r
\]
and $\delta_{r+1}$ is induced by $\delta_r$.
 \end{lem}
\begin{lem} \label{lem:3.5}
The $d_7$ differentials in the HFPSS for $\ctwo \wedge V(0)$ are generated by
\begin{align*}
    d_7(u^{-4})&=\alpha^7,\\
    d_7(u^{-5})&=u^{-1} \alpha^7,
\end{align*} 
and linearity with respect to $u_1$, 
$\alpha$ and $u^{\pm 8}$.
\end{lem}

\begin{proof}
The first differential follows from the same differential in the homotopy fixed point spectral sequence for $E^{hC_2}$. 
The spectral sequence $E_r(E^{hC_{2}}\wedge V(0))$ is a module over the ring spectral sequence $E_r(E^{hC_{2}})$, where $d_7(\alpha)=d_7(u_1)=0$ and $d_7(u^{\pm 8})=0$, hence the differentials are linear with respect to these elements. 
Finally, 
since $\delta_2(u^{-5})=u^{-4}\alpha$, 
\Cref{GBT} implies that 
 \[
 d_7(u^{-5})=d_7(\delta_7(u^{-4}\alpha ))=\delta_7(d_7(u^{-4}\alpha))=\delta_7(\alpha^8) =u^{-1}\alpha^7.
 \]
\end{proof}

\begin{figure}[h]
    \centering\includegraphics[width=0.8\linewidth]{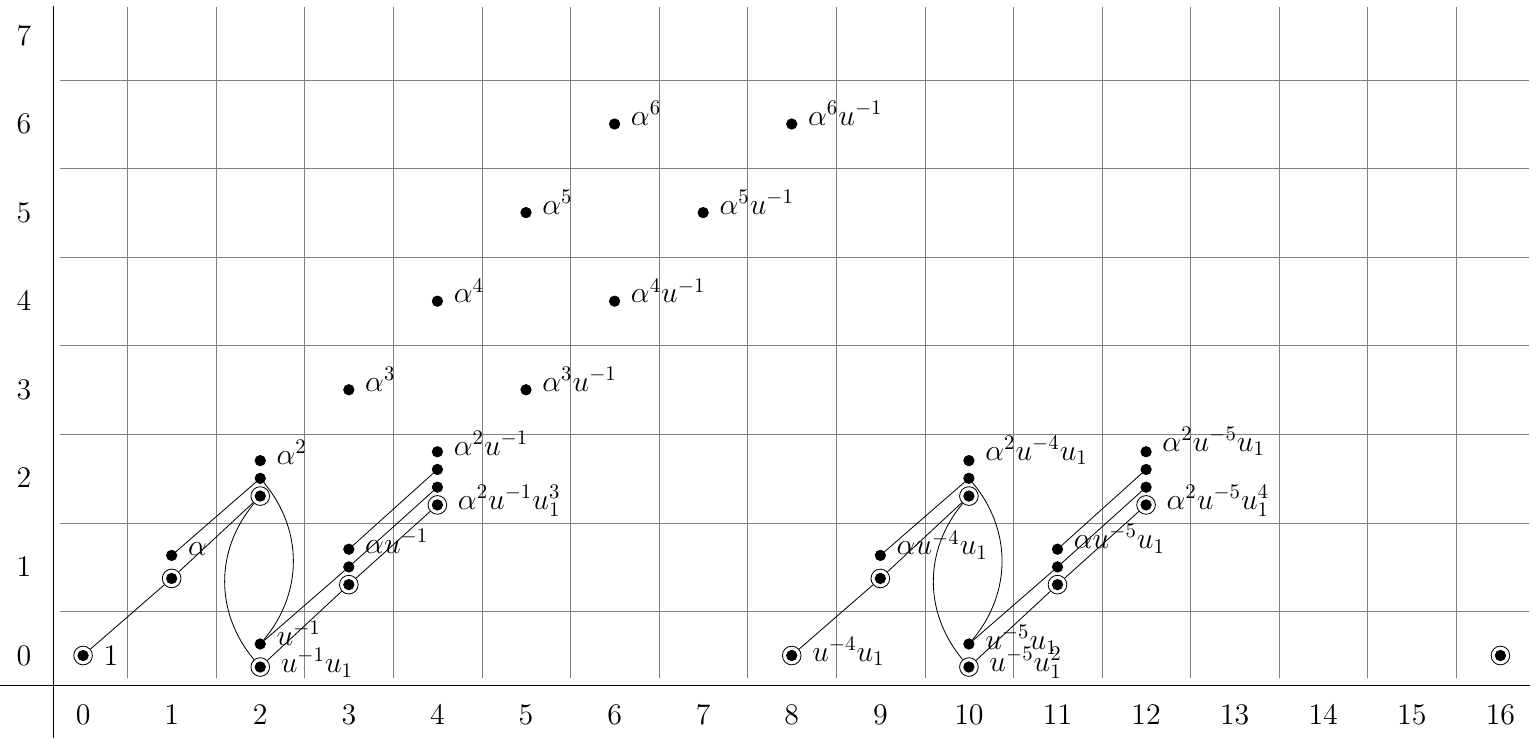}
    \caption{The homotopy groups of $\ctwo \wedge V(0)$. The notation is: $\bullet =\F_4$, and $\protect\obullet = \F_4\llbracket u_1\rrbracket$. The lines of slope 1 indicate multiplication by $\eta$. The lines connecting elements in the same stem indicate group extensions. The homotopy groups are 16-periodic.}
    \label{fig:Einfty_page_labeled}
\end{figure}

\subsection{Extension Problems}\label{subsec:c2smv0extproblems}
 Recall the long exact sequence in homotopy groups
\[
\cdots\to\pi_{2k-s}(\E^{hC_2}) \xrightarrow{i}\pi_{2k-s}(\E^{hC_2}\wedge V(0))\xrightarrow{p}\pi_{2k-s-1}(\E^{hC_2}) \xrightarrow{2}\cdots.
\]

All the non-trivial extensions on the $E_{\infty}$ page are generated by the extension in the following lemma. 
\begin{lem}\label{lem::c2_v0_extension}
We have $2u^{-1} = \alpha^2 u_1 \in \pi_2 (\E^{hC_2} \wedge V(0))$.
\end{lem}
\begin{proof}
For $u^{-1} \in \pi_2(\E^{hC_2}\wedge V(0))$, we have $p_\ast(u^{-1})=\alpha\in\pi_1 \E^{hC_2}$. Then by Lemma 2.19 from \cite{Beaudry_2022},
\[
2u^{-1}=i_\ast(\eta\alpha)=\alpha^2u_1.\qedhere
\]
\end{proof}

As a corollary of Lemma~\ref{lem::c2_v0_extension}, we are able to resolve the extension problems on the $E_{\infty}$ page (see Figure~\ref{fig:Einfty_page_labeled} and find that
$\pi_2(\ctwo\wedge V(0)) = \alpha^2\Fbb_4\oplus u^{-1}\W/4\llbracket u_1\rrbracket$ and $\pi_{10}(\ctwo\wedge V(0)) = \alpha^2 u^{-4} u_1 \Fbb_4\oplus u^{-5} u_1 \W/4\llbracket u_1\rrbracket $.
\section{The homotopy fixed point spectral sequence for \texorpdfstring{$\operatorname{E}^{hC_6}_2$}{E2hC6}}\label{sec:c6}
In this section, we compute the homotopy fixed point spectral sequence 
\begin{equation}\label{eq:C6-SS}
E^{s,t}_r(\E^{hC_6}): H^s(C_6, \E_t) \Longrightarrow \pi_{t-s}(\E^{hC_6}) 
\end{equation} 
We do not claim any originality for the computations in this section; they have been known for decades, starting with the work of Mahowald and Rezk in \cite{MR}. We present this computation here to build a base for our computations in the next two sections. 
\subsection{Computing the \texorpdfstring{$E_2$}{E2}-page}\label{sec:5.1}
There is an action of the group $C_3=\mathbb{F}_4^{\times}=\langle \zeta\rangle$ on the spectrum $\E^{hC_2}$ and  $\E^{hC_6}\simeq (\E^{hC_2})^{hC_3}$. In addition, the homotopy fixed point spectral sequence for $\csix$ is the $C_3$ fixed points of the homotopy fixed point spectral sequence for $\E^{hC_2}$,
with $d_r$ differentials in the former being the restriction of the  $d_r$ differentials in the latter. The action of $C_3$ on the group cohomology $H^*(C_2, \E_*)$ is given by
 (see (2.2) of \cite{1a7d99963b1048d8ad18fe824397caab}):
\begin{equation}\label{eq:C3action}
    \zeta \cdot u_1 = \omega u_1 \qquad
    \zeta \cdot u = \omega u \qquad
    \zeta \cdot \alpha = \omega^2\alpha,
\end{equation}
where $\omega$ is a primitive third root of unity.

\begin{figure}[h]
\includegraphics[width=0.9\linewidth]{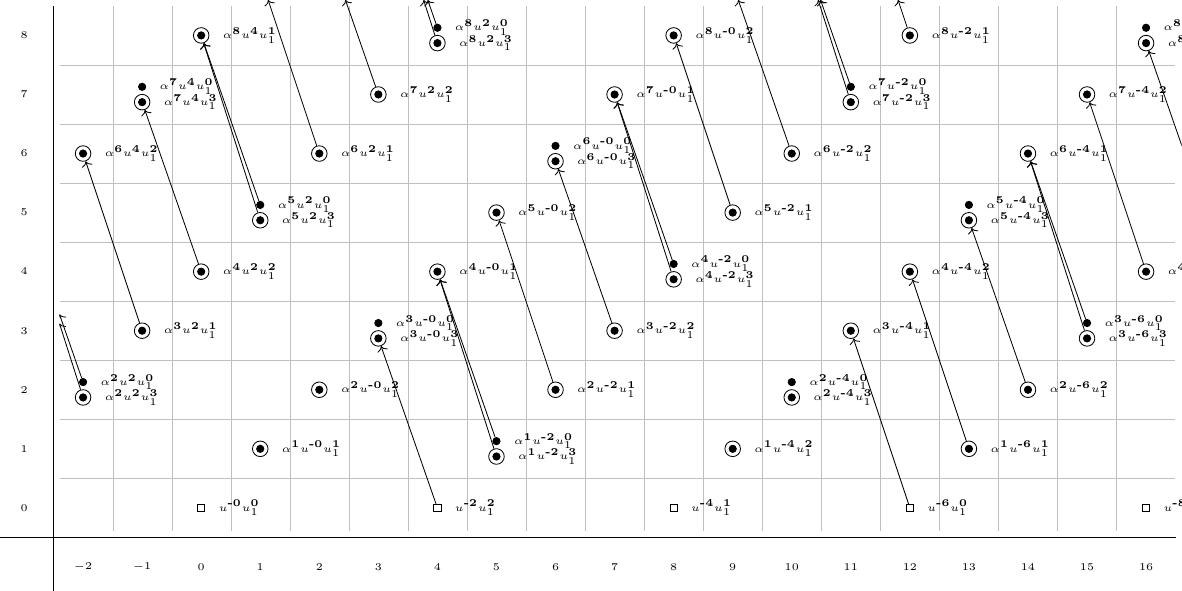}
\caption{The $E_3$ page of the HFPSS for $\csix$. The symbol $\bullet$ denotes $\Fbb_4$, the symbol $\Circled{\bullet}$ denotes $\Fbb_4 \llbracket u_1^3 \rrbracket$, and the symbol $\Box$ denotes $\W\llbracket u_1^3 \rrbracket$. The terms on the right of each symbol denote the generators.}\label{fig:e3_csix}
 \end{figure}

Applying these formulas, we can compute the $E_2$ page of the HFPSS for $\csix$. 

\begin{prop}[Lemma 2.3 of \cite{1a7d99963b1048d8ad18fe824397caab}]\label{prop:e2_csix}
Let $w=u^{-2}\alpha$. The $E_2$-page of the homotopy spectral sequence \eqref{eq:C6-SS} is given by:
\begin{align*}
E_2^{*,*}=H^*(C_6, \E_*)
=\W \llbracket u_1^3 \rrbracket[w, [u_1u^{-4}], [u^6]^{\pm 1}]/ (2w).
\end{align*}
\end{prop}
We also give names to the following important $C_6$-invariant elements 
\begin{equation}\label{eq:C6cohElements}
  v_1v_2:=u_1u^{-4} \qquad
  v_2^2:=u^{-6} \qquad
v_1^2:= u_1^2u^{-2} =  [v_1v_2]^2 [v_2^2]^{-1} 
\nonumber 
\end{equation}
Note that these elements are indecomposable: $v_1=u^{-1}u_1$ and $v_2=u^{-3}$ are not $C_6$ invariant and are not elements in this cohomology ring. Note also that $\nu \in \pi_3\mathbb{S}$ is detected by $\alpha^3=w^3 v_2^2$.

\subsection{Computing the intermediate pages}
The only non-trivial differentials in the spectral sequence \eqref{eq:C6-SS} $E_r(\E^{hC_6})$ are $d_3$ and $d_7$. As we mentioned above, the differentials in $E_r(\E^{hC_6})$ are simply restrictions to the $C_3$ fixed points of the differentials in the spectral sequence $E_r(\E^{hC_2})$.

\begin{lem}\label{lem:e3-csix}
  The $d_3$-differentials in the spectral sequence \eqref{eq:C6-SS} are generated by 
  \begin{align*}
  &d_3(u^{-2}u_1^2)=\alpha^3 u_1^3\\
  &d_3(u^{-2}\alpha)=\alpha^4 u_1
  \end{align*}
  and linearity with respect to $\nu=\alpha^3$,  $\eta=\alpha u_1$, $u^{-4} u_1$ and $v_2^{\pm 4}=u^{\pm 12}$.
\end{lem}

\begin{figure}[h]
\includegraphics[width=\linewidth]{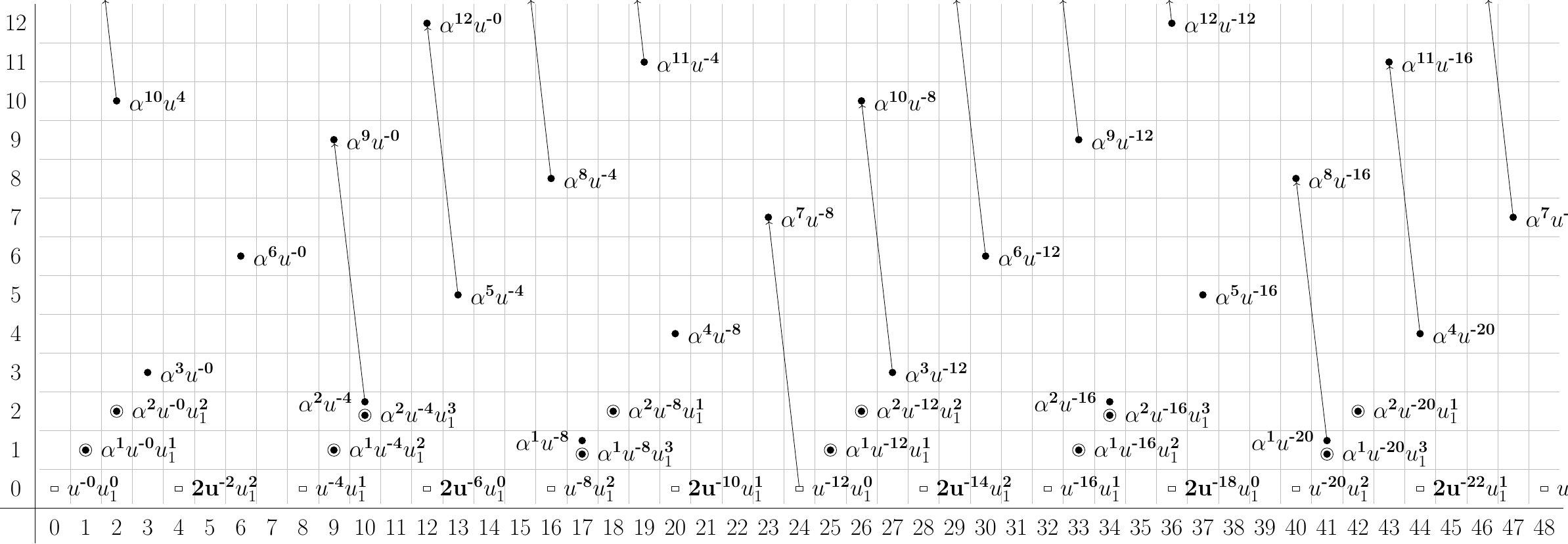}
\caption{The $E_7$ page of the HFPSS for $\csix$. The symbol $\bullet$ denotes $\Fbb_4$, the symbol $\Circled{\bullet}$ denotes $\Fbb_4 \llbracket u_1^3 \rrbracket$, and the symbol $\Box$ denotes $\W\llbracket u_1^3 \rrbracket$.}\label{fig:e7_csix}
\end{figure}

The $E_4$ page is (24,0)-periodic with the periodicity generator $u^{-12}$.
Due to sparseness, we have $E_4=E_7$ page, displayed in \Cref{fig:e7_csix}.
We summarize the differentials on the $E_7$ page in the following proposition.

\begin{prop}\label{prop:e7_csix}
    The $d_7$ differentials in the spectral sequence \eqref{eq:C6-SS} are generated by 
    \begin{align*}
    &d_{7}(\alpha^{2}u^{-4}) = \alpha^{9},\\
    &d_{7}(u^{-12}) = \alpha^{7} u^{-8},\\
    &d_{7}(\alpha u^{-20})= \alpha^{8} u^{-16},
    \end{align*}
    and linearity with respect to $\nu=\alpha^{3}$ and $v_2^{\pm 8}=u^{\pm 24}$.
\end{prop}

By sparseness, there are no differentials on pages $E_8$ and higher and $E_8=E_{\infty}$, displayed in \Cref{fig:e8_csix}. The homotopy groups are 48-periodic with periodicity generator $v_2^8$.
We record the following generators  in positive filtrations:
\begin{equation}\label{eq:c6names}
\eta=\alpha u_1  \qquad \overline{\kappa}=\alpha^4 u^{-8} \qquad y=\alpha u^{-8}=w v_2^2
\end{equation}
and note that $\nu=\alpha^3=y^3 v_2^{-8}.$

\begin{figure}[h]
\includegraphics[width=\linewidth]{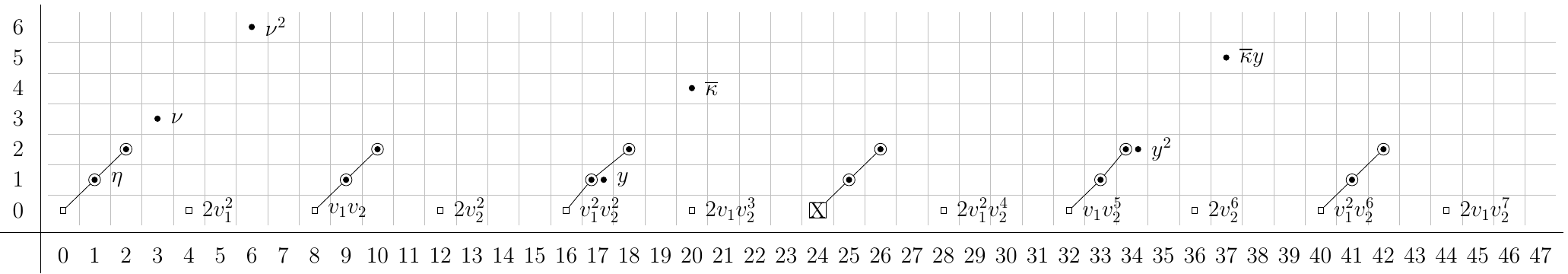}
\caption{The $E_8 = E_{\infty}$ page of the HFPSS for $\csix$. The symbol $\bullet$ denotes $\Fbb_4$,  $\Circled{\bullet}$ denotes $\Fbb_4 \llbracket u_1^3 \rrbracket$, and $\Box$ denotes $\W\llbracket u_1^3 \rrbracket$. The $\boxtimes$ at stem $24$ denotes $2v_2^4 \W \oplus v_1^3v_2^3 \W\llbracket u_1^3\rrbracket $. The lines denote multiplication by $\eta$. The homotopy groups are 48-periodic.}\label{fig:e8_csix}
\end{figure}
\section{The homotopy fixed point spectral sequence for \texorpdfstring{$\operatorname{E}^{hC_6}_2 \wedge V(0)$}{E2hC6 sm V(0)}}\label{sec:c6v0}
In this section, we compute the homotopy fixed point spectral sequence for $\csixv$ 
\begin{equation}\label{lem:c3_fix_points_V(0)}
E_2^{s,t}(\E^{hC_{6}}\wedge V(0)) \coloneqq (E_2^{s,t}(\E^{hC_{2}}\wedge V(0)))^{C_3} =H^s(C_6, \E_t/2) \Longrightarrow \pi_{t-s}(\E^{hC_6}\wedge V(0)).
\end{equation}

\begin{figure}
    \[\includegraphics[width=0.9\linewidth]{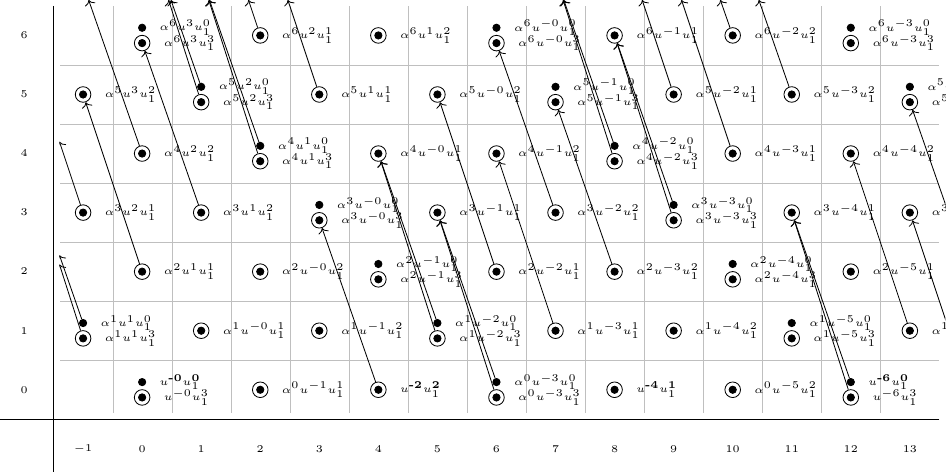}\]
    \caption{The $E_3$ page of HFPSS for $\csix \wedge V(0)$. The symbol $\bullet$ represents $\Fbb_4$, and the symbol $\Circled{\bullet}$ denotes $\Fbb_4 \llbracket u_1^3 \rrbracket$.}\label{fig::c6v0_e2}
\end{figure}
\begin{figure}[h]
\includegraphics[width=0.9\linewidth]{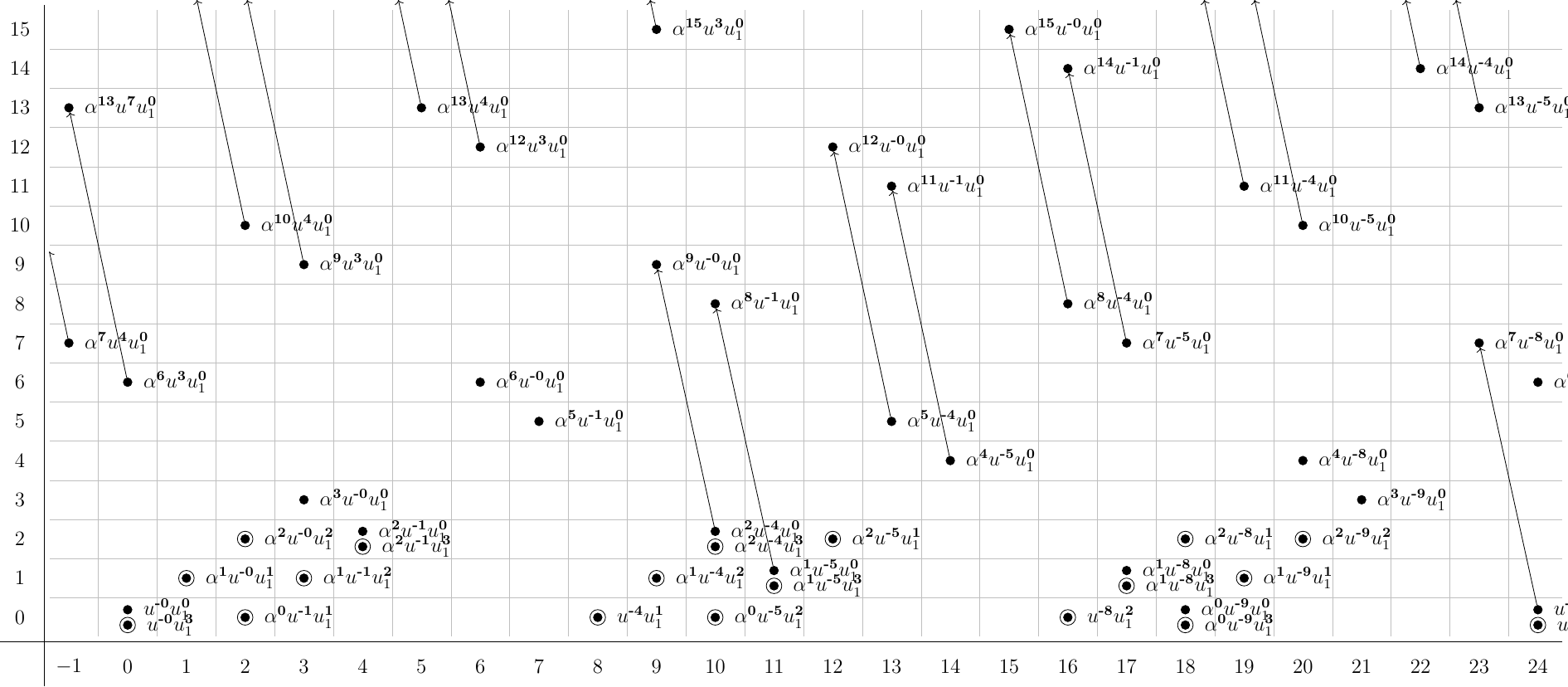}
\includegraphics[width=0.9\linewidth]{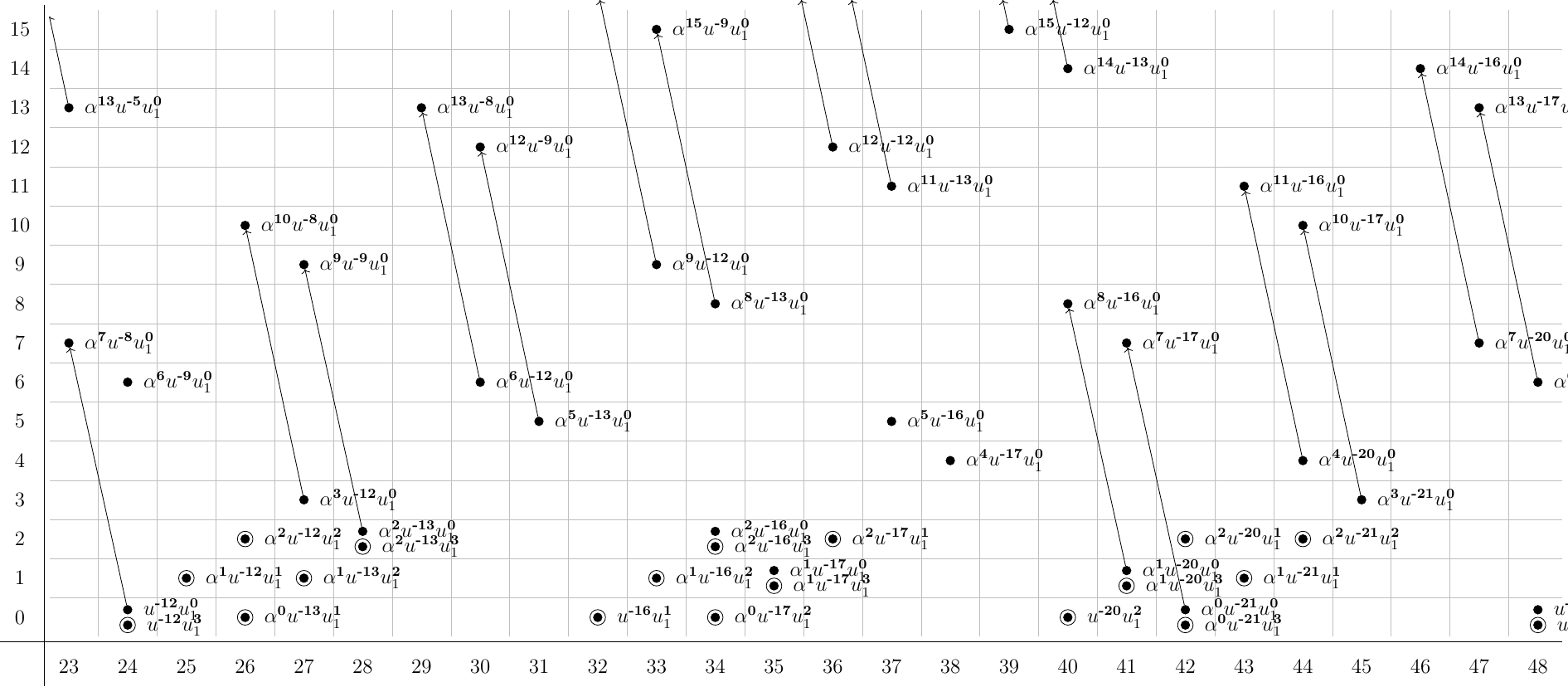}
\caption{The $E_7$ page of the HFPSS for $\csix \wedge V(0)$. The symbol $\bullet$ represents $\Fbb_4$, and the symbol $\Circled{\bullet}$ represents $\Fbb_4 \llbracket u_1^3 \rrbracket$.}\label{fig:c6v0_e7}
\end{figure}

\begin{figure}[h]
\includegraphics[width=0.9\linewidth]{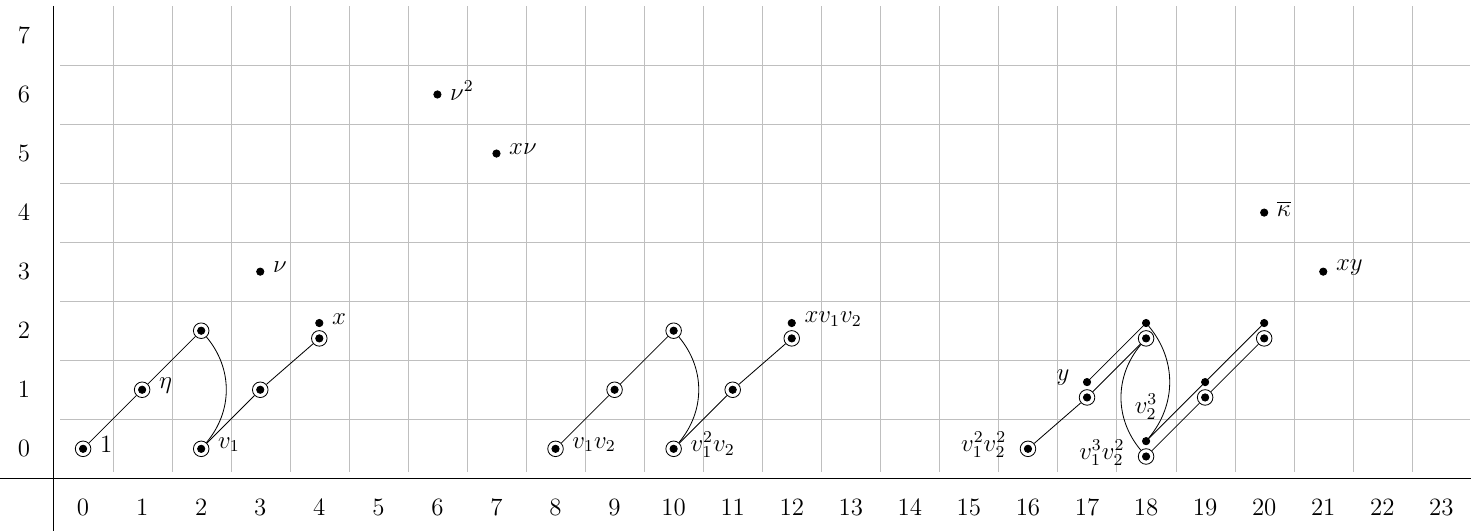}
\includegraphics[width=0.9\linewidth]{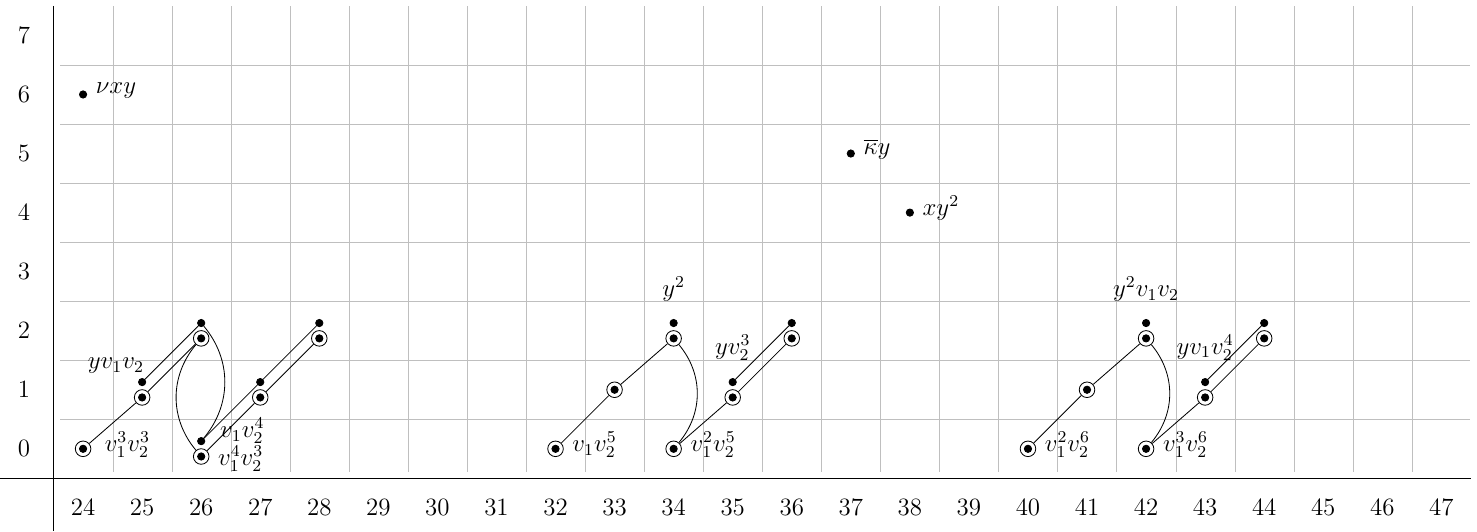}
\caption{The $E_8 = E_\infty$ page of the HFPSS for $\csixv$. The symbol $\bullet$ represents $\Fbb_4$, and the symbol $\Circled{\bullet}$ denotes $\Fbb_4 \llbracket u_1^3 \rrbracket$. The lines of slope $1$ denote multiplication by $\eta$.}\label{fig:c6v0_e8}
\end{figure}
\subsection{The $E_2$ page.}
Since  $\csix \simeq (\ctwo)^{hC_3}$, and $V(0)$ is a finite complex, we have  $\csixv \simeq (\ctwo \wedge V(0))^{hC_3}$. We use the isomorphism $H^{\ast}(C_{6}, \E_*/2) \cong H^{\ast}(C_{2}, \E_*/2)^{C_{3}}$ to compute $H^*(C_6, \E_*/2)$.
The homology of $H^{\ast}(C_{2}, \E_*/2)$ is described in \Cref{lem:E2C2V0} and the action of $C_3=\mathbb{F}_4^\times$ is as in \eqref{eq:C3action} and we read off the invariants.

\begin{lem}\label{lem:cohC6mod2}
We have
\[
H^*(C_6, \E_*/2)=\F_4[\![u_1^3]\!][[u_1u^{-1}], [u^{-3}]^{\pm 1}, w],
\]
where $w=\alpha u^{-2}$.
\end{lem}

\subsection{Differentials}
Most of the differentials follow from the fact that the HFPSS for $\csix \wedge V(0)$ is the fixed points of the HFPSS for $\ctwo \wedge V(0)$ under the action of $C_3$ given by \eqref{eq:C3action}.
\begin{lem}
The $d_3$ differentials in \eqref{lem:c3_fix_points_V(0)} are generated by 
\begin{align*}
&d_3(u_1^2u^{-2}) =\alpha^3 u_1^3\\
&d_3(u^{-3})=\alpha^3 u^{-1} u_1
\end{align*}
and linearity with respect to $\eta, \nu$, $u_1^3$ and $u^{\pm 12}$.
\end{lem}
\begin{proof}
   These differentials are the restrictions to $C_6$ fixed points of the differentials in \Cref{lem:d3c2v0}. 
\end{proof}

We will now compute the $d_7$ differentials in the following lemma.
\begin{lem}
    The  $d_7$ differentials in spectral sequence \eqref{lem:c3_fix_points_V(0)} are generated by
\[
\begin{aligned}
&d_7(\alpha^2 u^{-4})=\alpha^9\\
&d_7(u^{-12})= \alpha^7 u^{-8} \\
&d_7(\alpha u^{-20})= \alpha^8 u^{-16}
    \end{aligned}
    \qquad \qquad \qquad
    \begin{aligned}
    &d_7(\alpha u^{-5})=\alpha^8 u^{-1}\\
    &d_7(\alpha^2u^{-13})=\alpha^9 u^{-9}\\
    &d_7(u^{-21})=\alpha^7 u^{-17}
    \end{aligned}
    \]
    and linearity with respect to $\nu$, $u_1^3, \eta$ and $u^{\pm 24}$.
\begin{proof}
The differentials in the left column follow by naturality (or module structure over $\csix$) from the differentials in \Cref{prop:e7_csix}.  The differentials in the right column are restrictions to $C_6$ fixed points of the differentials in  \Cref{lem:3.5}.
Alternatively, the differentials in the right column follow by the Geometric Boundary Theorem (see \cite[Thm. 2.3.4]{ravgreen} and \cite[App. 4]{behrens_EHP} for the general statements and proofs or \cite[Thm. 2.17]{Beaudry_2022} for application to a similar case).
\end{proof}    
\end{lem}
A chart of the $E_7$ page is displayed in \Cref{fig:c6v0_e7} and $E_{\infty}=E_8$ page is displayed in \Cref{fig:c6v0_e8}. For the extensions, we use Lemma 2.19 from \cite{Beaudry_2022}, which shows that $2u^{\frac{j}{2}}{u_1}^{i} = \eta \alpha u^{\frac{j}{2}}{u_1}^{i}$. We use the following notation for the generators on the $E_{\infty}$ page (note that $y$ and $\overline{\kappa}$ are images of the like-named elements from $\pi_*E^{hC_6}$ \eqref{eq:c6names}):
\begin{equation}
    x=\alpha^2 u^{-1} \qquad v_1=u_1u^{-1} \qquad v_2=u^{-3}  \qquad y=\alpha u^{-8} \qquad \overline{\kappa}=\alpha^4u^{-8}
\end{equation}
\section{The homotopy fixed point spectral sequence for \texorpdfstring{$\operatorname{E}^{hC_6}_2 \wedge Y$}{pi\_*(C6 sm Y)}}\label{sec:c6y}
In this section, we compute the homotopy fixed point spectral sequence for $\csixy$ 
\begin{equation}
E_2^{s,t}(\E^{hC_{6}}\wedge Y) =H^s(C_6, \pi_{t}(\E \wedge Y)) \Longrightarrow \pi_{t-s}(\E^{hC_6}\wedge Y).
\end{equation}
\subsection{The $E_2$ page}
\begin{figure}[h]
    \includegraphics[width=0.9\linewidth]{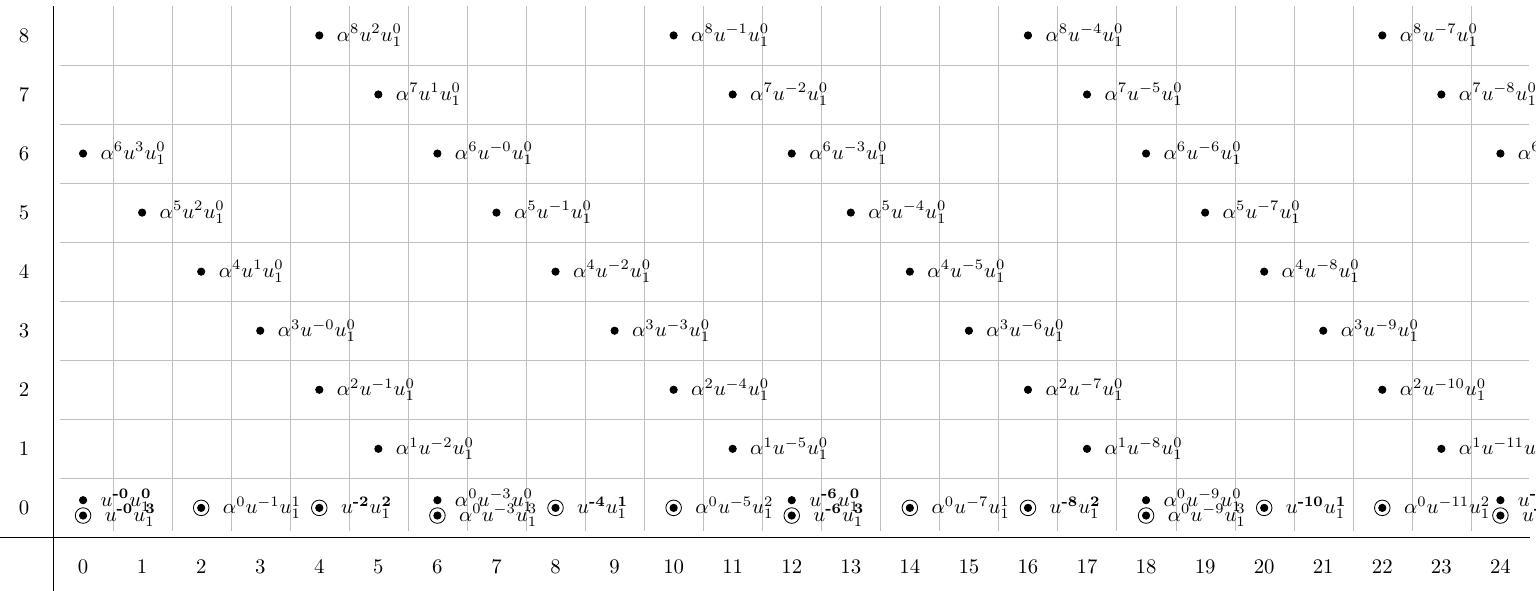}
    \caption{The $E_2$ page of HFPSS for $\csixy$. The symbol $\bullet$ represents $\Fbb_4$, and the symbol $\Circled{\bullet}$ denotes $\Fbb_4 \llbracket u_1^3 \rrbracket$.}\label{fig::c6y_e2}
\end{figure}

The cofiber sequence
\[
\csix \wedge\Sigma V(0) \xrightarrow{\eta} \csix \wedge V(0) \rightarrow \csix \wedge Y
\]
induces maps on the $E_2$ pages of the corresponding homotopy fixed point spectral sequences
\[
\cdots \xrightarrow{\eta} E_{2}^{s, t}(\E^{hC_{6}}\wedge V(0)) \xrightarrow{i} E_{2}^{s, t}(\E^{hC_{6}}\wedge Y) \xrightarrow{p} E_{2}^{s, t-2}(\E^{hC_{6}}\wedge V(0)) \xrightarrow{\eta} E_{2}^{s+1, t}(\E^{hC_{6}}\wedge V(0)) \rightarrow \cdots.
\]

Each multiplication by $\eta$ map above is injective, simply shifting degrees in a power series module. Hence the maps $p:E_{2}^{s, t}(E^{hC_{6}}\wedge Y) \rightarrow E_{2}^{s, t-2}(E^{hC_{6}}\wedge V(0))$  are zero maps. We conclude that each $E_{2}^{s, t}(E^{hC_{6}}\wedge Y)$ is isomorphic to the cokernel of $\eta$:
\[
0 \to E_{2}^{s-1, t-2}(\E^{hC_{6}}\wedge V(0)) \xrightarrow{\eta} E_{2}^{s, t}(\E^{hC_{6}}\wedge V(0)) \to E_{2}^{s, t}(E^{hC_{6}}\wedge Y) \to 0.
\]
The $E_2$ page is presented in \Cref{fig::c6y_e2}. 
\subsection{$d_3$ and $d_5$ differentials}
\begin{theo}\label{thm::c6y_diff_lin}
    All the differentials $d_r$ in the homotopy fixed point spectral sequence  for $\csix \wedge Y$ are linear with respect to $v_1 = u^{-1} u_1$, $\nu=\alpha^3$ and $v_2^{\pm 8}$.
\end{theo}

\begin{proof}
The spectrum $Y$ has a $v_1$-self map, hence 
Lemma 5.12 of \cite{Beaudry_2022} implies that all differentials $d_r$ on the HFPSS for $\csix \wedge Y$ are $v_1 = u^{-1} u_1$-linear. The differentials are $\alpha^3$-linear because $\alpha^3$ is a permanent cycle, as it detects $\nu \in \pi_{3}(\mathbb{S})$. The differentials are $v_2^{\pm 8}$ linear because of the module structure over the HFPSS for $\csix$. 
\end{proof}

Unlike the previous two spectral sequences, the differentials on the $E_3$-page are all trivial. 

\begin{lem} \label{lem:d3y}
    The $d_3$-differentials in the homotopy fixed point spectral sequence for $\csixy$ are all trivial.
\end{lem}
\begin{proof}
    The only groups that could potentially support a nontrivial $d_3$ differential are the groups $E_3^{0,t} = E_2^{0,t}$ for $t \equiv 4 \operatorname{mod} 6$. For these groups  we have \[E_3^{0, t} \cong u^{-\frac{t}{2}}u_1^2 \mathbb{F}_4\llbracket u_1^3\rrbracket \cong (u^{-3})^{\frac{t-4}{6}}v_1^2 \mathbb{F}_4\llbracket u_1^3\rrbracket.\] 
    Then we observe that for any $f\in\mathbb{F}_4\llbracket u_1^3 \rrbracket$, and any integer $n$ we have  $d_3(f(u^{-3})^n) =0$ due to sparseness. Then, since $d_3$ are $v_1$-linear (hence, $v_1^2$-linear), we have that $d_3$ must vanish for any element in $E_3^{0, t}$ for $t\equiv 4 \operatorname{mod} 6$. 
\end{proof}

\begin{figure}[h]
\includegraphics[width=0.9\linewidth]{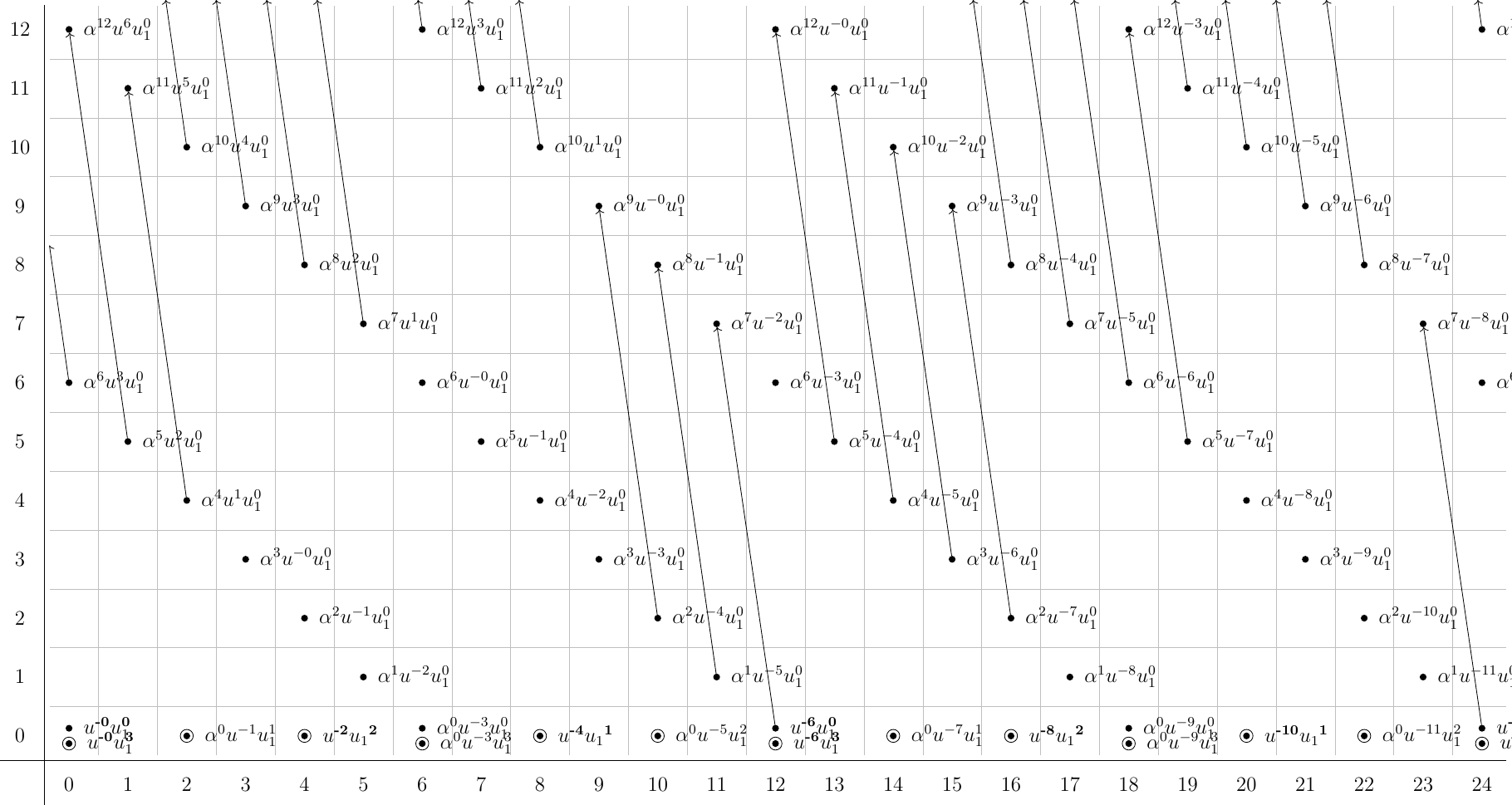}
\includegraphics[width=0.9\linewidth]{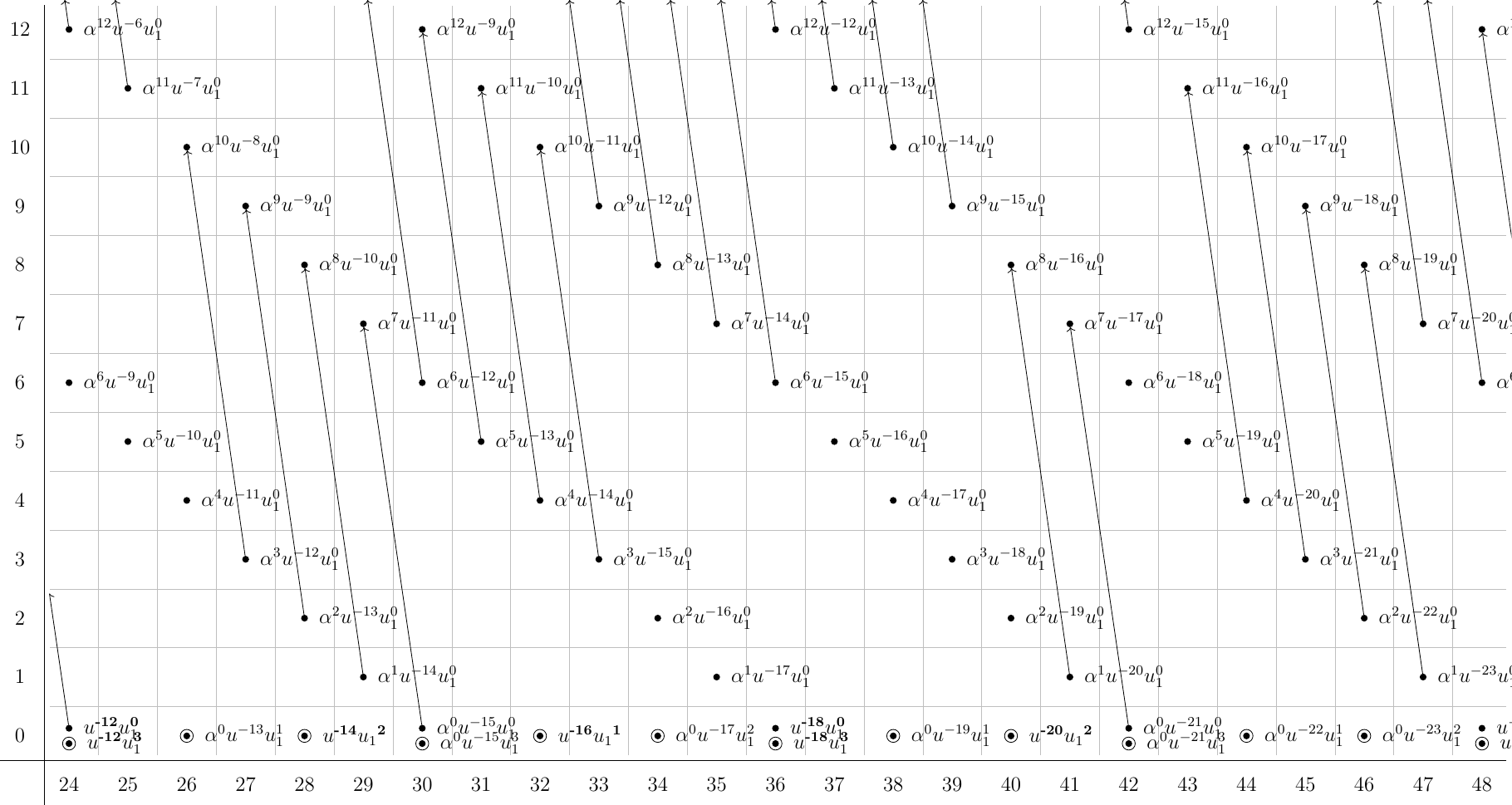}
\caption{The $E_7$-page of the HFPSS for $\csixy$. The symbol $\bullet$ represents $\Fbb_4$.}\label{fig:c6y_e7}
\end{figure}

\begin{lem}
    The $d_5$-differentials in the homotopy fixed point spectral sequence for $\csixy$ are all trivial. 
\end{lem}
\begin{proof}
    As in \Cref{lem:d3y}, the only possible sources for $d_5$ differentials are the groups $E_5^{0, t} = u^{-\frac{t}{2}}u_1\mathbb{F}_4\llbracket u_1^3\rrbracket$ for $t \equiv 2 \pmod 6$. For $f \in \mathbb{F}_4\llbracket u_1^3\rrbracket$, we have  $d_5(f u^{-\frac{t}{2}}u_1) = v_1d_5(f u^{-(\frac{t}{2} - 1)})  = 0$.
\end{proof}

    \begin{figure}[h]
\includegraphics[width=0.9\linewidth]{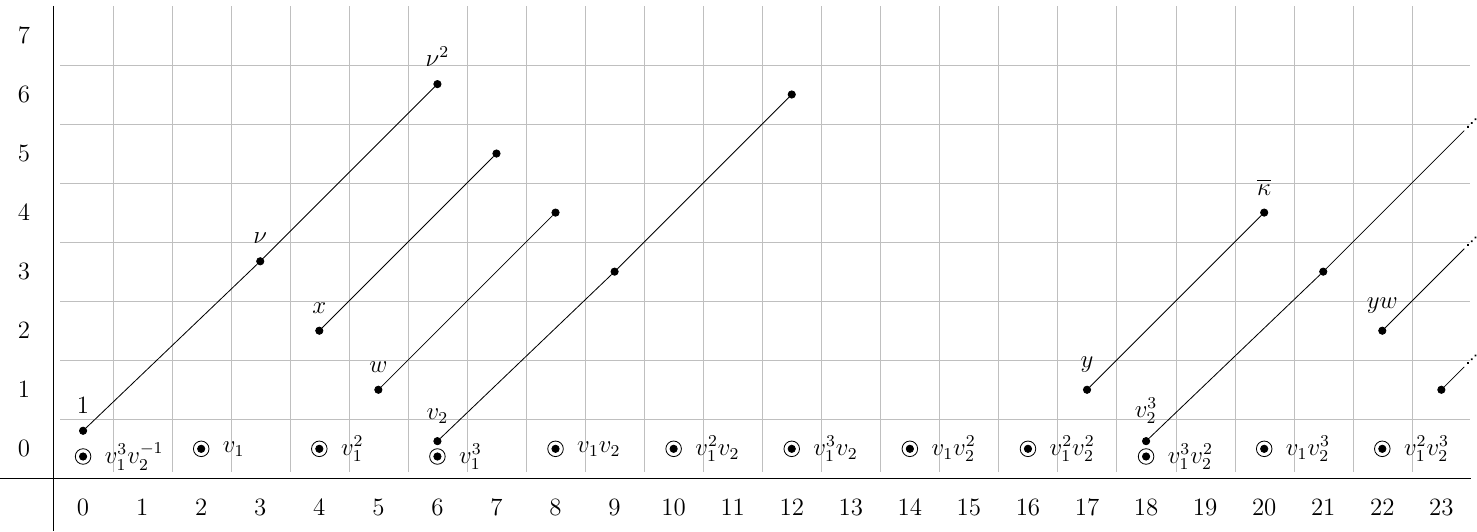}
\includegraphics[width=0.9\linewidth]{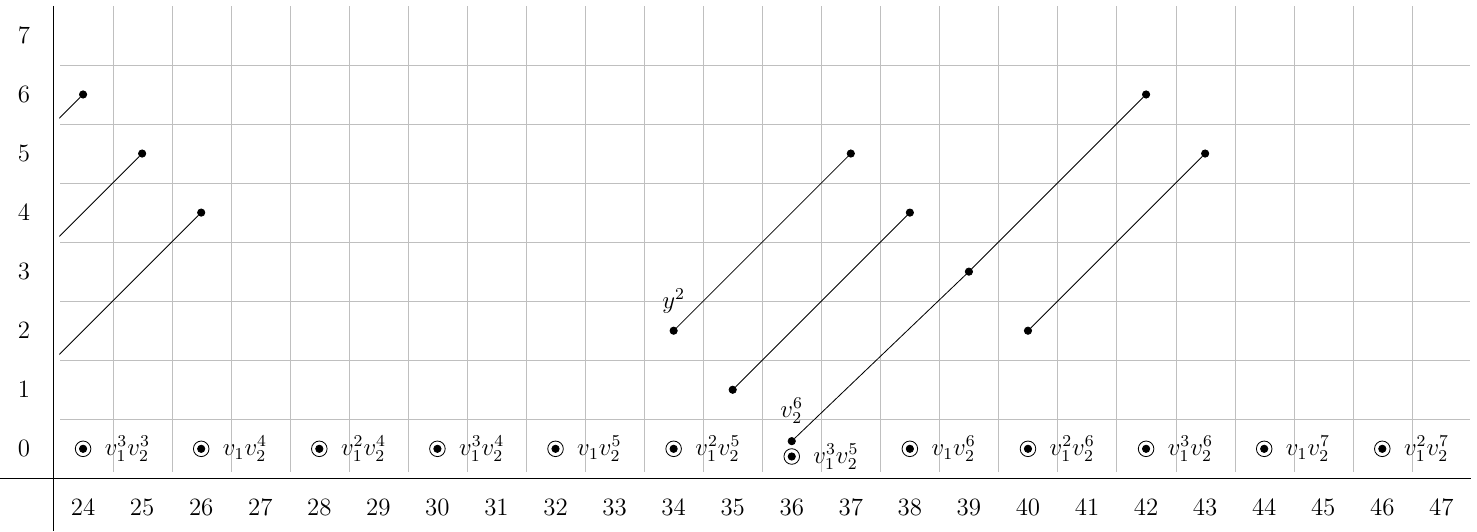}
\caption{The $E_{8} = E_{\infty}$ page of the HFPSS for $\csixy$. The symbol $\bullet$ represents $\Fbb_4$, and the symbol $\Circled{\bullet}$ denotes $\Fbb_4 \llbracket u_1^3 \rrbracket$. The lines represent multiplication by $\nu$. The homotopy groups are 48-periodic, with periodicity generator $v_2^8$.}\label{fig:c6y_cinfinity}
\end{figure}

\subsection{The \texorpdfstring{$d_7$}{d7}-differentials} 
Recall that the cofiber sequence 
\[
\Sigma \csix  \wedge V(0)\xrightarrow{\eta} \csix \wedge V(0) \rightarrow \csix \wedge Y
\]
induces a long exact sequence in homotopy groups
\begin{equation}\label{eq:LES-V-Y}
\cdots \rightarrow \pi_{k-1}(\csix \wedge V(0)) \xrightarrow{\eta} \pi_{k}(\csix \wedge V(0)) \rightarrow \pi_{k}(\csix \wedge Y) \rightarrow\cdots.    
\end{equation}

From this, we can gain information about some of the groups in $\pi_{\ast}(\csix \wedge V(0))$, which will help us compute the $d_7$ differentials.

\begin{lem}\label{lem::c6y_e7_gen}
The $d_7$ differentials in the HFPSS for $\csixy$ are generated by
\[
\begin{aligned}
&d_7(\alpha^2 u^{-4}) = \alpha^{9}\\
&d_7(u^{-12})=\alpha^7 u^{-8}\\
&d_7(\alpha u^{-20})= \alpha^8 u^{-16}
\end{aligned}
\qquad 
\begin{aligned}
&d_7(\alpha u^{-5})= \alpha^{8} u^{-1}\\
&d_7(\alpha^2u^{-13})=\alpha^9u^{-9}\\
&d_7(u^{-21})= \alpha^7 u^{-17}
\end{aligned}
\qquad \qquad 
\begin{aligned}
 &d_7(u^{-6})=\alpha^7 u^{-2}\\
 &d_7(\alpha u^{-14})=\alpha^8 u^{-10}\\
 &d_7(\alpha^2 u^{-22})= \alpha^9 u^{-18}
\end{aligned}
\qquad
\begin{aligned}
 &d_7(\alpha^2 u^{-7})=\alpha^9 u^{-3}\\
 &d_7(u^{-15})=\alpha^7 u^{-11}\\
 &d_7(\alpha u^{-23})= \alpha^8 u^{-19}
\end{aligned}
\]
and linearity with respect to $\nu=\alpha^3$ and $u^{\pm 24}$.
\end{lem}

\begin{proof}
  The differentials in the two columns on the left follow from naturality and differentials in  HFPSS for $E^{hC_6} \wedge V(0)$.  

  The simplest way to deduce the other differentials is to compute some of the homotopy groups $\pi_i \csix \wedge Y$ using \eqref{eq:LES-V-Y} and analyze the $E_7$ page of the homotopy fixed point spectral sequence to see what differentials will need to be non-trivial to ensure those homotopy groups values. 
  
  First,  observe that $\pi_{15} (\csix \wedge Y) =0$. The only way to make that happen in the homotopy fixed  point spectral sequence is to have $d_7(u^{-6}\alpha^3)=\alpha^{10}u^{-2}$ and $d_7(\alpha^2  u^{-7})=\alpha^9 u^{-3}$. The former differential and $\alpha^3$-linearity then imply $d_7(u^{-6})=\alpha^7 u^{-2}$.  

  Next, $\pi_{29}(\csix \wedge Y)=0$ implies the differentials $d_7(\alpha u^{-14})=\alpha^8 u^{-10}$ and $d_7(u^{-15})=\alpha^7 u^{-11}$. 

  Finally, $\pi_{45}(\csix \wedge Y)=\pi_{47}(\csix \wedge Y)=0$ implies the differentials $d_7(\alpha^2 u^{-22})=\alpha^9 u^{-18}$ and $d_7(\alpha u^{-23})=\alpha^8 u^{-19}$. 
\end{proof}

A chart of the $E_7$ page can be found in \Cref{fig:c6y_e7}.
Due to sparceness, there are no $d_r$ for $r>7$, so we have $E_8 = E_{\infty}$.
It remains for us to resolve the extensions.
    By Lemma 2.19 of \cite{Beaudry_2022}, any element divisible by two is in the image of $\eta$, hence every group extension on the $E_{\infty}$ page splits.

\FloatBarrier

\bibliographystyle{alpha}
\bibliography{references.bib}

\end{document}